\documentclass[a4paper,12pt]{extarticle}
\usepackage[utf8]{inputenc} 		
\usepackage{amsmath,amssymb,amsthm}	
\usepackage{graphicx}			
\usepackage{hyperref}			
\usepackage{color}			
\usepackage[left=30mm, right=40mm, top = 30mm, bottom = 30mm]{geometry}
\DeclareMathOperator{\secc}{secc}
\DeclareMathOperator{\Ricc}{Ric}
\DeclareMathOperator{\grad}{grad}
\DeclareMathOperator{\Div}{div}
\DeclareMathOperator{\Hess}{Hess}

\title{ Lower bound for the first eigenvalue of a minimally embedded  hypersurface in a Riemannian manifold}
\author{Egor Surkov\footnote{Department of Higher Geometry and Topology, Faculty of Mathematics and Mechanics, Moscow State University, Leninskie Gory, GSP-1, 119991, Moscow, Russia \\ \tt{egor.surkov@math.msu.ru} }}
\date{}
\begin{document}
\maketitle									
\begin{center}
\textbf{Abstract}
\end{center}
We provide a lower bound  for the first eigenvalue of the Laplace-Beltrami operator on a closed orientable hypersurface minimally embedded in an orientable compact  Riemannian manifold with  Ricci curvature bounded below by a positive constant.
\section{Introduction}
Let $M$ be an orientable compact manifold of dimension $n+1$ without boundary with a Riemannian metric $\overline{g}$, and let  $\Sigma$ be a minimally embedded oriented closed hypersurface $\psi :\Sigma \hookrightarrow M$. Consider the metric $g=\psi^{*}\overline{g}$ induced on $\Sigma$ by the embedding. Then we have the Laplace-Beltrami operator on $\Sigma$,
\begin{equation}
\Delta^{\Sigma}_{g}f = -\Div\,\grad f =  -\frac{1}{\sqrt{|g|}} \frac{\partial}{\partial x^{i}}\left(\sqrt{|g|}  g^{ij} \frac{\partial f}{\partial x^{j}}\right) .
\end{equation}
Let us denote the first non-zero eigenvalue of the operator $\Delta^{\Sigma}_{g}$ by $\lambda_{1}(\Sigma,g)$. 

An important problem in Spectral Geometry and Geometric Analysis is to find bounds for the spectrum of Laplace-Beltrami operator in terms of intrinsic or extrinsic geometry of (sub)manifolds. There are several famous works in the case of intrinsic geometry, see  e.g \cite{Cheeger} and \cite{Buser}, that give lower and upper bounds for the first eigenvalue in some geometric terms. In the case of surfaces a significant progress was achieved  in the problem of isoperimetric inequalities for eigenvalues. Upper bounds depending only on topological data for normalized eigenvalues  were given in the works \cite{YiYi}, \cite{Korevaar}, \cite{Karp2} that were later refined. Hence, it is natural to look for a maximal metric for a normalized eigenvalue, see \cite{Hersh}, \cite{LYi}, \cite{Nadir}, \cite{KleinB1}, \cite{KleinB2}, \cite{genus2}. It is important to remark, that this problem was solved completely for all eigenvalues in the case of $\mathbb{S}^{2}$ and $\mathbb{RP}^2$ in the papers \cite{KNPP} and \cite{Karp}. More about critical metrics for eigenvalues can be found in the reviews \cite{Uspekhi}, \cite{TrMatNauk}.

In this work we are interested in the extrinsic case. In particular, the question about the connection between the first eigenvalue of minimally embedded hypersurfaces and the geometry of the ambient manifold is an old  problem. One of the most famous conjectures   is the following Yau conjecture \cite{Yau}. 
\newtheorem*{conjecture}{Yau conjecture}
\begin{conjecture}
Let $M = \mathbb{S}^{n+1}$, the unit sphere, and $\Sigma$ is a minimally embedded oriented closed hypersurface. Then the first non-zero eigenvalue of $\Sigma$ is equal to $n$.
\end{conjecture}
Yau's conjecture is still open, but there are several papers where the conjecture is proven in special cases, see e.g. \cite{ChoiSorrer}, \cite{Symmetric21}, \cite{Symmetric1}.

Let us fix notation and conventions. The sectional curvature of the manifold $M$ is denoted by $\secc_{M}$, the Ricci curvature tensor is denoted by $\Ricc_{M}$, the Riemann curvature tensor is denoted by $$R_{X,Y}Z = \nabla_X\nabla_Y Z - \nabla_Y\nabla_X Z - \nabla_{[X,Y]} Z.$$ We  denote the second fundamental form by $B\in \Gamma(Hom(T\Sigma \otimes T\Sigma , N\Sigma))$. We also need the mean curvature normal field $\overrightarrow{H_\Sigma} = \sum\limits_{i=1}^{n}B(e_{i},e_{i})$, where $e_{i}$ is a local orthonormal frame in the sections of $T\Sigma$. We also define mean curvature $H_\Sigma$ by a formula $\overrightarrow{H_\Sigma} = H_\Sigma\nu$, where $\nu$ is a normal field to $\Sigma$. We consider $M$ and $\Sigma$ with a chosen orientation, this choice becomes clear in  Section 3. The projections of the field $X \in \Gamma(TM)$ onto the tangent and the normal bundles of the submanifold $\Sigma$ is denoted by $X^T$ and $X^N$, respectively. We need the Hessian $\Hess w = \nabla dw \in \Gamma(TM \otimes TM)$, where $w \in C^{\infty}(M)$. Differential operators on $\Sigma$ have $\Sigma$ as a superscript, for example, the Laplacian on $M$ will be written as $\Delta$, and the Laplacian on $\Sigma$ will be written as $\Delta^{\Sigma}$.
 
The formula $\Ricc_{M} \geqslant k$ means that for any $s \in \Gamma(TM)$ we have $\Ricc_{M}(s,s) \geqslant k\cdot \overline{g}(s,s)$. 

Approaching the Yau conjecture, Choi and Wang proved the following result. Let $\Ricc_{M} \geqslant q > 0$. It is shown in the paper \cite{ChoiWang} that in this case one has
\begin{equation}\label{Formule1}
\lambda_{1}(\Sigma)  \geqslant \frac{q}{2}.
\end{equation}

Remark that in the case of the unite sphere $\mathbb{S}^{n+1}$ it is known, that the first eigenvalue of $\Sigma$ is less or equal to $n$. And it is clear that the finding a good lower bound is enough to prove the Yau conjecture. In a case similar to Choi-Wang setting there are some results  concerning an upper bound for the first eigenvalue, see \cite{Heintz}.

In the case when the ambient manifold is the Euclidean space there are  famous Reilly's results and their various generalizations, see \cite{UppRei}, and in the case when ambient manifold is a space form,  see \cite{SpaceForm}.

Further, in the case when $M = \mathbb{S}^{n+1}$, the unit sphere, it is shown  in the work \cite{SyreSpruckDuncan}, using the methods of the theory of tubular neighborhoods, that there exist such constants
\begin{equation}
    a_{n} \geqslant \frac{(n-1)n^{2}}{3200},
  \qquad
   b_{n} \leqslant \frac{5n^{2}}{720},
 \end{equation}
that the first eigenvalue of the hypersurface $\Sigma$ is bounded as follows,
\begin{equation}\label{Formule2}
\lambda_{1}(\Sigma) \geqslant \frac{n}{2} + \frac{a_n}{\Lambda^6 + b_n},
\end{equation}
where $\Lambda =  \sup\limits_{\Sigma}{\vert \vert B \vert \vert}$.

In this paper we step aside from the context of the Yau conjecture and  we generalize the result of \cite{SyreSpruckDuncan} to a wider class of manifolds. We assume that $M$ satisfies the following condition,
\begin{equation}\label{Condition}
   c= \inf_{\substack{p\in M \\ \sigma \in T_{p}M}}\secc_{M}(\sigma) > 0,
 \end{equation}
 where $\sigma$ is a two-dimensional plane in $T_{p}M$.

 From condition \eqref{Condition} it  follows immediately that $ \Ricc_M \geqslant cn$. Let
 \begin{equation}\label{Corrolary}
 b= \sup_{\substack{p\in M \\ \sigma \in T_{p}M}}\secc_{M}(\sigma).
 \end{equation}
 The main result of this paper is the following theorem.
 \newtheorem*{theorem}{Theorem 1.1.1}
 \begin{theorem} Let $M$ be a compact oriented Riemannian manifold of dimension $n+1$ without boundary  satisfying condition \eqref{Condition}, and let $\Sigma$ be a closed orientable hypersurface minimally embedded in $M$.
 Then there exists such constants  \begin{align}
  \begin{aligned}
\varepsilon \in \left(0\,,\,\frac{\Lambda}{2}\right),
  \end{aligned}
  &&
  \begin{aligned}
\beta > 0,
  \end{aligned} 
  \end{align}
  that for 
 \begin{align}
   \begin{aligned}
 \widetilde{\varepsilon} = \left( \frac{\varepsilon}{\sqrt{b} - \sqrt{c}\frac{\varepsilon}{\Lambda}}\right) \left(1 + \frac{bn}{\Lambda^2}\right)\sqrt{c},
 \end{aligned}
&&
  \begin{aligned}
\delta = \frac{c\,n}{\sqrt{b}}\arctan\left(\frac{\varepsilon}{n\sqrt{c}}\right),
  \end{aligned}
 \end{align}
 we have 
  
  \begin{align}
\gamma = \sqrt{2nc} -\widetilde{\varepsilon}-\beta > 0,
   \end{align}
   and for constants defined as
   \begin{align}
  \begin{aligned}
  a_n=\frac{(cn-1)\delta^3\gamma}{64},
  \end{aligned}
  &&
  \begin{aligned}
  b_n=\frac{(cn-1)\delta^3}{32\beta},
  \end{aligned}
 \end{align}
  the following lower bound holds,
 \begin{equation}\label{Main}
\lambda_{1}(\Sigma) \geqslant \frac{cn}{2} + \frac{a_n}{\Lambda^6 + b_n},
\end{equation}
where $\Lambda = \sup\limits_{\Sigma}{\vert \vert B \vert \vert}$.
 In addition, we can choose  $\beta$ and $\varepsilon$ in such a way that  the following estimates holds,
 \begin{equation}\label{Formule267}
a_n \geqslant \frac{3(cn-1)(c \cdot n)^{7/2}}{b^{3/2}\,6400}\arctan^3\left(\frac{\varepsilon}{\sqrt{c} \,n}\right),
\end{equation}
\begin{equation}\label{Formule268}
b_n \leqslant \frac{5(cn-1)(c\cdot n)^{5/2}}{b^{3/2}\,8}\arctan^3\left(\frac{\varepsilon}{\sqrt{c}\, n}\right),
\end{equation}
\begin{equation}\label{Formule269}
 \text{where  }  \varepsilon= \frac{2(2\sqrt{b}-\sqrt{c})\sqrt{n}}{3\sqrt{b}\frac{bc+1}{b}}.
\end{equation}
\end{theorem}

\section{Preliminaries}
\subsection{Parallel hypersurfaces}
For a hypersurface $\Sigma$ embedded in $M,$ consider the restriction $exp^{N\Sigma}$ of the exponential map from $TM$ to $N\Sigma \subset TM$. We define the following family of maps 
\begin{equation}\label{Formule3}
\Phi_t: \Sigma \longmapsto M, \quad \Phi_t(p) = exp^{N\Sigma}(p,t\nu_{p}) \in M,
\end{equation}
where $\nu$ is the unit normal field to the surface $\Sigma$, induced by the orientation of the surface. The choice of orientation on $\Sigma$ and $M$ is described in Section 3. 

Let $\varkappa_{i}$ be the principal curvatures of the surface $\Sigma$. We introduce $\varkappa_{max} = \max\limits_{p\in \Sigma, \, i \in 1...n}\{\varkappa_{i}|_{p}\}$ and $\varkappa =\frac{1}{ \sqrt{c}}\arctan(\sqrt{c}\,\varkappa_{max}^{-1})$. Then the general theory of tubular neighborhoods \cite[Lemma 8.20]{Gray} says that for any $t \in {(-\varkappa,\varkappa)}$ the map $\Phi_{t}$ is an immersion.

We denote by $\varkappa_{i}(t)$  the $i$-th principal curvature of the surface $\Phi_t(\Sigma)$. It is clear  that $\varkappa_{i}(0)$ coincides with $\varkappa_{i}$.

Moreover, it is known that for any $ t_{1} \in {[0,\varkappa)}$ the principal curvatures of the hypersurface $\Phi_{t_1}(\Sigma)$ at a point  $\Phi_{t_1}(p)$ has the following upper bound \cite[Lemma 8.21]{Gray},
\begin{equation}
\label{Napp}
\varkappa_{i}(t_1)|_{\Phi_{t_1}(p)} \leqslant \frac{\sqrt{b}\cdot \tan(\sqrt{b} \, t_1)+ \varkappa_{i}|_{p}} {1- \frac{\varkappa_{i}|_{p}}{\sqrt{b}}\cdot \tan(\sqrt{b} \, t_1)}.
\end{equation}

In order to simplify notation, we  denote the image of $\Phi_t(\Sigma)$ by $\Sigma_t$. Unfortunately, in the general case, it is not true that the mapping $\Phi_t$ is an embedding for all values of $t$ for which it is defined. However, if we impose restrictions on the curvature of the surface $\Sigma$, then the following lemma holds.

Here mean convexity means that $H \geqslant 0$ and strictly mean convexity means that $H > 0$.

\newtheorem*{lemma}{Lemma 2.1.1} 
\begin{lemma} Let $\Sigma \hookrightarrow M$ be an embedded closed mean-convex oriented hypersurface and $\secc_M \geqslant c > 0$, then for any $t \in (-\varkappa;\varkappa)$, where $\varkappa =\frac{1}{ \sqrt{c}}\arctan(\sqrt{c}\,\varkappa_{max}^{-1})$, the hypersurface $\Sigma_{t}$ is embedded, and for $t \in {(-\varkappa;\varkappa) \setminus \{0\}}$ it is strictly mean-convex.
\end{lemma}
\begin{proof}
The proof is in three steps and by reductio ad absurdum.

The first step. 
Let 
\begin{equation}
\label{TubeCond}
t_{*}=\sup\{t \in [0,\varkappa) | \, \forall \tau \in [0,t)  \, \text{the hypersurface} \,  \Sigma_\tau \, \text{is embedded}  \}.
\end{equation}
It is known that for small $t$ the surface $\Sigma_t$ is embedded \cite[Chapter 2]{Gray}, it implies that $t_{*} > 0$. If the assertion of the lemma is not true, then $t_{*} < \varkappa$ and there exists $\varepsilon \in (0,1)$ such that $t_{*}=\frac{1}{\sqrt{c}}\arctan(\varepsilon\sqrt{c}\varkappa_{max}^{-1})$. Since $t_{*} < \varkappa$, it is clear that  $\Phi_{t_{*}}$ is an immersion.

Note that injectivity of an immersion is equivalent to fact that this immersion is an embedding, if we consider the  immersions of compact manifolds. Indeed, a map from a compact space to a Hausdorff space is always closed. A closed and injective map defines a homeomorphism onto its image. Thus, for compact manifolds, the injectivity of an immersion is equivalent to fact that this immersion is an embedding. Hence, if we assume that our immersion  is no more an embedding, then it is no more injective.

 It follows from the definition of $t_{*}$ and the previous discussion, that there exists at least one point $\ x\in \Sigma_{t_{*}}$ such that for some $p,q \in \Sigma$ we have $x= exp^{N\Sigma}(p,t_{*}\nu_{p})=exp^{N\Sigma}(q,t_{*}\nu_{q})$.

Remark that there exist neighborhoods  $U_p$ and $U_q$ of points $p$ and $q$ such that  $\Phi^{p}_{t_{*}} = \Phi_{t_{*}}|_{U_p}$ and $\Phi^{q}_{t_{*}} = \Phi_{t_{*}}|_{U_q}$ satisfy the following properties: 1) the images of these neighborhoods by the  mappings $\Sigma^{p}_{t_*}$ and $\Sigma^{q}_{t_*}$ have at least one common point $x$ and 2) $\Sigma^{p}_{t_*}$ and $\Sigma^{q}_{t_*}$ are graphs over the corresponding neighborhoods in tangent spaces. The second property means that we can choose neighborhoods  $U_p$ and $U_q$ in such a way  that  there exist neighborhoods  in the tangent spaces $\widetilde{U_p} \subset T_p\Sigma$ and $\widetilde{U_q} \subset T_q\Sigma$  diffeomorphic by maps $exp^{\Sigma}_p$ and  $exp^{\Sigma}_q$ to $U_p$ and $U_q$, and the maps $\Phi^{p}_{t^*}$ and $\Phi^{q}_{t_{*}}$  can be considered as defined  on the neighborhoods  $\widetilde{U_p}, \widetilde{U_q}$ identified with $U_p$ and $U_q$.

For an embedded surface $S$ with a chosen unit normal field we consider the signed distance function to the surface $S$. It is well defined in a sufficiently small neighborhood of $S$ see \cite[Chapter III.6]{Chavel}. We denote this function by $\rho_{S}(x)$. Locally the surface is the set of zeros of this function. Note that for an embedded hypersurface  Fermi coordinates are correctly defined on its sufficiently small  neighborhood. It follows that $\nu = \grad \rho_{S}|_{S}$,  for more details see e.g. \cite[Chapter III.6]{Chavel}. Further, for an embedded hypersurface we  consider  the normal field $\nu$  extended as $\grad \rho_{S}$ into some neighborhood in $M$ of the surface $S$.

The second step. Consider the surfaces $\Sigma^{p}_{t_*}$ and $\Sigma^{q}_{t_*}$. 
We want to prove that $\Sigma^{p}_{t_*}$ and $\Sigma^{q}_{t_*}$  touch each other at the point $x$, and their normal vectors  have opposite directions at the point $x$.

First, note that $M$ is a complete metric space. There exists a value  $\alpha > 0$ such that the following condition holds: the surfaces $\Phi_{t_{*} - \alpha}(U_p) = \Sigma_{t_{*} - \alpha}^p$ and $\Phi_{t_{*} - \alpha}(U_q) = \Sigma_{t_{*} - \alpha}^q$ lie in a normal neighborhood in $M$ of the point $x$ and in  Fermi neighborhoods of each of the surfaces $\Sigma_{t_{*} - \alpha}^p$ and $\Sigma_{t_{*} - \alpha}^q$. This can be achieved by choosing sufficiently small $\alpha > 0$ and reducing if necessary  the neighborhoods $\widetilde{U_p}, \widetilde{U_q}$, since the condition holds  for $\alpha =0$  and it is clear that that it also holds for all positive $\alpha$ sufficiently close to zero.

Let $\Phi_{t_{*} - \alpha}(p) = p_\alpha$ and $\Phi_{t_{*} - \alpha}(q)= q_\alpha$. By the triangle inequality we have that
$$d(\Sigma_{t_{*} - \alpha}^p, \Sigma_{t_{*} - \alpha}^q) \leqslant d(p_\alpha, q_\alpha)\leqslant d(p_\alpha,x) + d(q_\alpha,x) = 2\alpha,$$
where the last equality is satisfied since we work in the normal neighborhood of $x$ and we already have the geodesics of corresponding lengths, since for $\alpha=0$ such geodesics exist and this property depends continuously w.r.t. $\alpha$.

Suppose that $d(\Sigma_{t_{*} - \alpha}^p, \Sigma_{t_{*} - \alpha}^q) < 2\alpha$. We want to obtain a contradiction. We consider all geodesics as parameterized by their oriented lengths.

Recall that  Fermi coordinates are given by a diffeomorphism from a neighborhood of $\Sigma_{t_{*} - \alpha}^q$ in the normal bundle to a neighborhood $U \in M$ of $\Sigma_{t_{*} - \alpha}^q$ in the manifold,
$$ \mathcal{F}_{t_{*} - \alpha}^q : B_{r}N\Sigma_{t_{*} - \alpha}^q \xrightarrow{exp^{N\Sigma_{t_{*} - \alpha}^q}} U.$$
Moreover, $\mathcal{F}_{t_{*} - \alpha}^q$ is consistent with $\rho_{\Sigma_{t_{*} - \alpha}^q}$, i.e. in the neighborhood $U$ with  the Fermi coordinates, the geodesics generated by the mapping $exp^{N\Sigma_{t_{*} - \alpha}^q}$ coincide with the gradient trajectories of the function $\rho_{\Sigma_{t_{*} - \alpha}^q}$ see \cite[Chapter III.6]{Chavel}. We can assume that the diffeomorphism $\mathcal{F}_{t_{*} - \alpha}^q$ associates a pair to each point $b$ in the neighborhood of $U$. This pair consists of point $c$ on the surface $\Sigma_{t_{*} - \alpha}^q$ and the distance from the point $b$ to the surface $\Sigma_{t_{*} - \alpha}^q$, furthermore that distance is equal to $d(b,c)$. Moreover, this distance is realized by the gradient trajectory of the function $\rho_{\Sigma_{t_{*} - \alpha}^q}$, which is a minimal geodesic.

Note that the set $A := \{a \in \Sigma_{t_{*} - \alpha}^p \, | \,d(a,\Sigma_{t_{*} - \alpha}^q) < 2\alpha \}$ has interior points. Then, since the surface $\Sigma_{t_{*} - \alpha}^p$ lies in the Fermi neighborhood of the surface $\Sigma_{t_{*} - \alpha}^q$ and $\mathcal{F}_{t_{*} - \alpha}^q$ is open, we can choose an interior point $\bar{p}_{\alpha} \, \in A$, such that $\pi_{t_{*} - \alpha}^q \circ {\mathcal{F}_{t_{*} - \alpha}^q}^{-1}(\bar{p}_{\alpha})= \bar{q}_{\alpha} \in \Sigma_{t_{*} - \alpha}^q$, where $\pi_{t_{*} - \alpha}^q$ is the projection from $B_{r}N\Sigma_{t_{*} - \alpha}^q$ onto $\Sigma_{t_{*} - \alpha}^q$.

 Consider the geodesic $\gamma_1$ that connects $\bar{q}_{\alpha}$ and  $\bar{p}_{\alpha}$ which  existence is provided by  the properties of Fermi coordinates  described above. This geodesic minimizes the distance between $\bar{p}_{\alpha}$ and $\Sigma_{t_{*} - \alpha}^q$. Then the length  of  the arc of $\gamma_1$ between $\bar{q}_{\alpha}$ and $\bar{p}_{\alpha}$ is less than $2\alpha$. We consider the geodesic $\gamma_1$  parametrized in such a way that  $\gamma_1(0)=\bar{q}_{\alpha}$ and   $\gamma_1(d(\bar{q}_{\alpha},\bar{p}_{\alpha}))=\bar{p}_{\alpha}$. Note that since $\gamma_1$ is the minimizer of distance between $\bar{p}_{\alpha}$ and $\Sigma_{t_{*} - \alpha}^q$ , and we are working in the Fermi neighborhood of the surface $\Sigma_{t_{*} - \alpha}^q$, then $\dot{\gamma_1}(0)$ is  normal  to the surface $\Sigma_{t_{*} - \alpha}^q$, see \cite[Chapter III.6]{Chavel}. Let $|\gamma_1|$ denote the length arc of the curve $\gamma_1$ between $\bar{q}_{\alpha}$ and  $\bar{p}_{\alpha}$, and let $e = \frac{|\gamma_1|}{2}$.

The Fermi coordinates in the neighborhood  of $\Sigma_{t_{*} - \alpha}^p$ are defined by the open mapping $\mathcal{F}_{t_{*} - \alpha}^p$. The mapping $\mathcal{F}_{t_{*} - \alpha}^p$ is defined in a  similar to $\mathcal{F}_{t_{*} - \alpha}^q$ way. We need the midpoint $m_{\gamma_1} = \gamma_1(e)$ of the geodesic $\gamma_1$. The point $m_{\gamma_1}$ is an interior point for the domain of the Fermi coordinates with respect to the surface $\Sigma_{t_{*} - \alpha}^p$.  Then we obtain the point $\pi_{t_{*} - \alpha}^p \circ {\mathcal{F}_{t_{*} - \alpha}^p}^{-1}(m_{\gamma_1}) = h \in \Sigma_{t_{*} - \alpha}^p$, where $\pi_{t_{*} - \alpha}^p$ is defined in a similar to  $\pi_{t_{*} - \alpha}^q$ way. The same arguments as for $\gamma_1$ yield the existence of $\gamma_2$ which connects $h$ and $m_{\gamma_1}$ and it is the minimizer of distance  between $m_{\gamma_1}$ and $\Sigma_{t_{*} - \alpha}^p$. We can write $d(m_{\gamma_1},\Sigma_{t_{*} - \alpha}^p)=|\gamma_2|$, where $|\gamma_2|$ denotes the arc length of the curve $\gamma_2$ connecting the points $h$ and $m_{\gamma_1}$. We consider $\gamma_2$ parametrized as $\gamma_2(0) = h$ and $\gamma_2(d(m_{\gamma_1},h))= m_{\gamma_1}$.    The same argument as for $\gamma_1$ implies that $\dot{\gamma_2}(0)$ is the normal vector to the surface $\Sigma_{t_{*} - \alpha}^p$. Since $\gamma_2$ is the minimizer of distance between $m_{\gamma_1}$ and $\Sigma_{t_{*} - \alpha}^p$ and we can consider a half of $\gamma_1$ as the arc connecting $m_{\gamma_1}$ and $\Sigma_{t_{*} - \alpha}^p$, we obtain that $|\gamma_2| \leqslant e < \alpha$ by the construction.

Consider the mapping $\Phi_{t_{*} - (\alpha- e)}$. If $|\gamma_2| = e$, then  we obtain from the properties of $\gamma_1$ and $\gamma_2$ that
$$\Phi_{t_{*} - (\alpha- e)}(\Phi_{t_{*} - \alpha}^{-1}(\bar{q}_{\alpha}))= m_{\gamma_1} = \Phi_{t_{*} - (\alpha- e)}(\Phi_{t_{*} - \alpha}^{-1}(h)),$$
and $\Phi_{t_{*} - (\alpha- e)}$ is no more injective, it contradicts the definition of $t_{*}$ .
Therefore, we can assume that $|\gamma_2| < e$. Consider the surfaces $\Sigma^{p}_{t_*- (\alpha-e)}$ and $\Sigma^{q}_{t_*- (\alpha-e)}$ which are the images of the neighborhoods $U_p$ and $U_q$ under the mapping $\Phi_{t_{*} - (\alpha-e)}$. Then $m_{\gamma_1} \in \Sigma^{q}_{t_*- (\alpha-e)}$. Consider the point $\bar{q} = \Phi_{t_{*} - (\alpha-e)}(q)$. Recall that we consider the signed  distances to the hypersurface. Then, from the definition of the point $q$  we have that $\rho_{\Sigma^{p}_{t_*- (\alpha-e)}}(\bar{q}) > 0$. Since $\dot{\gamma_2}(0)$ is the normal vector to the surface $\Sigma_{t_{*} - \alpha}^p$, it follows that  $\gamma_2$ was generated by the map $\Phi_{t_{*} - (\alpha- e)}$. From the previous discussion and from the inequality  $|\gamma_2| < e$, we finally obtain that $\rho_{\Sigma^{p}_{t_*- (\alpha-e)}}(m_{\gamma_1}) < 0$. Take any path $\Gamma$ on the surface $\Sigma^{q}_{t_*- (\alpha-e)}$ connecting the points $\bar{q}$ and $m_{\gamma_1}$. Then there exists a point $w$ such that $\rho_{\Sigma^{p}_{t_*- (\alpha-e)}}(w)=0$, and hence at this point $\Phi_{t_{*} - (\alpha-e)}$ is not injective, this contradicts the choice of $t_{*}$.

Thus, the distance between $\Sigma_{t_{*} - \alpha}^p$ and $ \Sigma_{t_{*} - \alpha}^q$ is $2\alpha$. Recall that it follows from the definition of the point $x$ that $d(p_\alpha,x) + d(q_\alpha,x)=2\alpha$. We have a piecewise smooth curve consisting of two geodesics connecting the points $p_\alpha,x$ and the point $x,q_\alpha$, respectively. A piecewise smooth curve that is the minimizer of distance is the shortest curve, and therefore is a minimal geodesic and therefore is smooth. We obtain that $\Sigma^{p}_{t_*}$ and $\Sigma^{q}_{t_*}$ are tangent at the point $x$, and their normal vectors have opposite directions. The last statement follows from the fact that the normal vectors at point $x$ will be the velocity vectors of one geodesic, parametrized in different directions.

The third step. Now recall \cite[Lemma 8.20]{Gray}, which says that the mean curvature of the surface $\Sigma_{\tau}$ can be bounded from below in terms of the principal curvatures of the surface $\Sigma$ in the following way. There exists $\delta \in [0;1)$ depending on $\tau$ such that 
\begin{equation}
\begin{gathered}
\label{Bond}
H_{\Sigma_{\tau}}  \geqslant  \sum_{i=1}^n \frac { \varkappa_{i} + \sqrt{c}\,\tan(\sqrt{c}\,\tau)}{1- \frac{1}{\sqrt{c}}\tan(\sqrt{c}\,\tau)\varkappa_{i}} \geqslant  \,H_{\Sigma}+ \sum_{i=1}^n \frac{(\sqrt{c}+{\sqrt{c}}^{-1}\varkappa_{i}^2)\tan(\sqrt{c}\, \tau)}{1-\delta\cdot\varepsilon}  >  \\ > \sum_{i=1}^n \frac{(\sqrt{c}+{\sqrt{c}}^{-1}\varkappa_{i}^2)\tan(\sqrt{c}\, \tau)}{1-\delta\cdot\varepsilon} = \frac{(\sqrt{c}n+{\sqrt{c}}^{-1}\vert \vert B \vert \vert^2)\tan(\sqrt{c}\,\tau)}{1-\delta\cdot\varepsilon} > 0,
\end{gathered}
\end{equation}

It follows from the Weingarten formula for hypersurfaces and the divergence formula in an local orthonormal frame that 
for a hypersurface $S$ embedded in $M$ and for a point $l \in S$  the following formula holds
\begin{equation}\label{Formule6}
H_{S}(l)= -\sum_{i=1}^{n} \langle \nabla_{e_i}\nu; e_i\rangle = -\sum_{i=1}^{n} \langle \nabla_{e_i}\nu; e_i\rangle - \langle\nabla_{\nu}\nu; \nu\rangle = -\Div \nu(l).
\end{equation}
 We used here the fact that $\langle \nabla_{\nu}\nu; \nu\rangle = \frac{1}{2} \nu\langle\nu; \nu\rangle = \frac{1}{2} \nu1 = 0$.

Locally one has $\nu=\grad \rho_S$. Hence, formula \eqref{Formule6} can be rewritten as
\begin{equation}\label{Formule7}
H_{S}(l)=\Delta \rho_S(l).
\end{equation}

Consider again our surfaces $\Sigma^p_{t_{*}} $ and $\Sigma^q_{t_{*}} $. The mean curvatures of these surfaces at the point $x$ are positive, which follows from the estimate \eqref{Bond} of the mean curvature at the point $x$. It follows immediately, from formula \eqref{Formule7} applied to the surfaces $\Sigma^p_{t_{*}} $ and $\Sigma^q_{t_{*}} $, that
\begin{equation}\label{Formule8}
-\Delta \rho_{\Sigma^p_{t_{*}}}(x)<0<\Delta \rho_{\Sigma^q_{t_{*}}}(x).
\end{equation}
Then, by continuity, for a sufficiently small geodesic ball $B(x)$ with center at the point $x$, the following inequalities hold,
\begin{equation}\label{F358}
-\Delta \rho_{\Sigma^p_{t_{*}}}<0<\Delta \rho_{\Sigma^q_{t_{*}}}.
\end{equation}

It follows that $-\Delta (\rho_{\Sigma^p_{t_{*}}} + \rho_{\Sigma^q_{t_{*}}}) < 0$. In the domain $D=B(x)\cap\{\rho_{\Sigma^q_{t_{*}}}>0\}$ we have $\rho_{\Sigma^p_{t_{*}}} + \rho_{\Sigma^q_{t_{*}}} \geqslant 0$ and $\rho_{\Sigma^p_{t_{*}}}(x)=\rho_{\Sigma^q_{t_{*}}}(x)=0$. By the Hopf lemma \cite[Lemma 6.4.2]{Evans}, applied at the point $x$ as a boundary point of $D$, we obtain that either $\rho_{\Sigma^p_{t_{*}}}+\rho_{\Sigma^q_{t_{*}}}=c$ in  $D$, or $\nu_{D}(\rho_{\Sigma^p_{t_{*}}}+\rho_{\Sigma^q_{t_{*}}})(x) > 0$, where  $\nu_{D}$ denotes the exterior unit normal field to the boundary of the domain $D$ in $M$. Note that at the point $x$ we have that $\nu_{D}(x)=\nu(x)$, there $\nu$ indicates the unit normal field to $\Sigma^p_{t_{*}}$. The first formula given by the Hopf lemma  contradicts the strictness of inequality \eqref{F358}, and the second formula given by the Hopf lemma contradicts the fact that $\nu \rho_{\Sigma^p_{t_{*}}}(x) = -\nu \rho_{\Sigma^q_{t_{*}}}(x)=1$. This contradiction proves Lemma 2.1.1.
\end{proof}

\subsection{Simons Theory}

In this subsection we use the results first appeared in the paper \cite{Simons}, but we want to use the notation from \cite{Xin} . In particular, let
\begin{equation} \label{tildeB}
\widetilde{\mathcal{B}} = B \circ B^{*} \circ  B,
\end{equation}
where $B^{*} \in \Gamma(Hom(N\Sigma,T\Sigma \otimes T\Sigma))$ is the formal conjugate to $B \in \Gamma(Hom(T\Sigma \otimes T\Sigma , N\Sigma))$,
\begin{equation} \label{tildeR}
\widetilde{\mathcal{R}}(X,Y) = \sum_{i=1}^{n} [(\nabla^{M}_X R^{M})_{Y, e_{i}} e_i + (\nabla^{M}_{e_{i}} R^{M})_{X,e_{i}} Y]^{N},
\end{equation}
where $X,Y \in T\Sigma$ and $e_{1}...e_{n}$ is a local orthonormal frame in $T\Sigma$,
\begin{equation} \label{underR}
\begin{gathered}
\underline{\mathcal{R}}(X,Y) = \sum_{i=1}^{n} \, [ 2R^{M}_{Y, e_{i}}B(X,e_i) + 2R^{M}_{X, e_{i}}B(Y,e_i) - B(X,{R^{M}_{Y, e_{i}}e_i}^{T}) - \\
- B(Y,{R^{M}_{X, e_{i}}e_i}^{T}) + R^{M}_{B(X,Y),e_i}e_i - 2B(e_i,{R^{M}_{X, e_{i}}Y}^{T})]^{N}.
\end{gathered}
\end{equation}

Let us write the formula \cite[formula 1.6.4]{Xin} in the case of codimension~1,
\begin{equation}\label{Formule3333}
tr(\nabla^2 B)= -\widetilde{\mathcal{B}} + \widetilde{\mathcal{R}} + \underline{\mathcal{R}},
\end{equation} 
where
$$tr(\nabla^2 B)= \sum_{i=1}^n \nabla_{e_i}\nabla_{e_i} B - \nabla_{\nabla_{e_i} e_i} B,$$ 
and $e_{1}...e_{n}$ is a local orthonormal frame in $T\Sigma$.
Let us multiply the equality \eqref{Formule3333} by $B$ as an element of $\Gamma(T\Sigma \otimes T\Sigma , N\Sigma)$ using the natural scalar product. Estimating from below following the works of \cite{Xin} and \cite{Simons}, we obtain that
\begin{equation}\label{Formule334}
\langle tr(\nabla^2 B), B\rangle \geqslant -|B|^4 +  nc|B|^2+ \langle \widetilde{\mathcal{R}},B\rangle.
\end{equation}

Integrating both parts, we obtain
\begin{equation}\label{Formule335}
\int_{\Sigma} \langle tr(\nabla^2 B), B\rangle ds^g \geqslant \int_{\Sigma} |B|^2(nc - |B|^2) ds^g + \int_{\Sigma} \langle \widetilde{\mathcal{R}},B\rangle ds^g.
\end{equation}
From the definition of $ \widetilde{\mathcal{R}}$ in the case of codimension 1 and from Gauss–Codazzi equations, it follows that $\int_{\Sigma} \langle \widetilde{\mathcal{R}},B\rangle ds^g~\geqslant~0$, see \cite{Simons}. This implies that
\begin{equation}\label{Formule336}
 \int_{\Sigma} \langle tr(\nabla^2 B), B\rangle ds^g  \geqslant \int_{\Sigma} |B|^2(nc - |B|^2) ds^g.
\end{equation} 

We use the equality $ \int_{\Sigma} \langle tr(\nabla^2 B), B\rangle ds^g = - \int_{\Sigma}|\nabla B|^2 ds^g$ from the book \cite[§1.6]{Xin}, where the norm of RHS is taken with respect to the scalar product in $\Gamma(Hom(T\Sigma \otimes T\Sigma \otimes T\Sigma, N\Sigma))$. Next, we obtain
\begin{equation}\label{Formule337}
0 \geqslant - \int_{\Sigma}|\nabla(B)|^2 ds^g \geqslant \int_{\Sigma} |B|^2(nc - |B|^2) ds^g.
\end{equation}

It follows that either
\begin{equation}\label{N666}
|B| \geqslant \sqrt{nc},
\end{equation}
or the hypersurface $\Sigma$
is totally geodesic. We exclude the case of a totally geodesic submanifold, since its sectional curvature will be equal to the sectional curvature of the manifold\,$M$. This follows from the fact that $\secc_\Sigma(u,v) = K_{exp(u,v)}$, where $K$ is the Gaussian curvature of the surface generated by the vectors $u$ and $v$ $\in T_{p}\Sigma$ under the exponential map. Since the surfaces are totally geodesic, the images of the corresponding surfaces coincide, so the Gaussian curvatures coincide too. Then it is easy to see that $\Ricc_\Sigma \geqslant c(n-1)$. Let us recall the Lichnerowicz theorem, see \cite{Lichnerowicz}.

\newtheorem*{Theorem}{Lichnerowicz's Theorem}
\begin{Theorem}
Let $V$ be an orientable compact manifold of dimension $n$ without boundary  with Ricci curvature bounded below as
$$\Ricc_V \geqslant (n-1)C \geqslant 0,$$ then the first positive eigenvalue of the Laplace-Beltrami operator is bounded below,
\begin{center}
$\lambda^{V}_{1} \geqslant nC$.
\end{center}
\end{Theorem}

Applying this theorem to $\Sigma$, we obtain in the case of a totally geodesic hypersurface $\Sigma$ that
\begin{equation}\label{Formule333}
\lambda_1 \geqslant nc,
\end{equation}
which is much better than the estimate we want to obtain.

Therefore, we exclude the case of totally geodesic hypersurfaces. Further  we assume that for the hypersurfaces $\Sigma$ we have  $\Lambda = \sup\limits_{\Sigma}{\| B \|} \geqslant \sqrt{nc}$.

\section{Bounds for eigenvalues}
\subsection{The result by Choi and Wang}

In this section, we recall the proof of the Choi and Wang inequality. In the following sections we improve this inequality in order to prove our main result \eqref{Main}.

We need the following formula first obtained in \cite{Reilly} in order to present the results of \cite{ChoiWang}.
\newtheorem*{lemma8}{Proposition 3.1.1}
\begin{lemma8}[Reilly's Formula] Let $(X,g)$ be a smooth oriented Riemannian manifold with boundary $\partial X$ and the unit normal field to the boundary directed inward. Let $dv^{g}$ be the volume form on $X$, and let $ds^{g}$ be the volume form on $\partial X$. Let $H_{\partial X}$ be the mean curvature of $\partial X$ consistent with the orientation of $\partial X$. Then the following formula holds,
\begin{equation}
\begin{gathered}
\begin{gathered}
\label{Raileigh1}
\int_{X}((\Delta u)^2- \|\Hess u\|^2)dv^g = \int_{X}\Ricc_{X}(\grad u,\grad u)dv^g + \\ +  \int_{\partial X}(\Delta^{\partial X}u - Hu_{\nu})u_\nu ds^g 
+ \int_{\partial X}\langle \grad^{\partial X} u , \grad^{\partial X} u_\nu \rangle ds^g -  \\ - \int_{\partial X}B(\grad^{\partial X} u , \grad^{\partial X} u)ds^g,
\end{gathered}
\end{gathered}
\end{equation}
where $u_{\nu} = \nu u$.
\end{lemma8}

Consider again $\Sigma$ and $M$ from the previous sections. It is important to note that since $\Ricc_{M} \geqslant cn$,  the first Betti number $b_1(M) = 0$, see Bochner \cite{Bochner}. From the Poincaré duality and the fact that $H^1(M)= Hom(\pi_1(M),\mathbb{Z})$, see \cite{Hatcher}, we can conclude  that there is no torsion in the corresponding (co)homology group of oriented manifolds.  Combining the previous two facts we obtain $H_{n}(M) = 0$.  Applying the long exact sequence of pair and the excision theorem to $\Sigma \hookrightarrow M$, we obtain that $\Sigma$ divides $M$ into two domains. Let us denote them by $M_1$ and $M_2$. Further, we assume that the normal field on $\Sigma$ is chosen in such a way that $ - \int_{\Sigma} B(\grad^{\Sigma} u , \grad^{\Sigma} u)ds^g \geqslant 0$. Without loss of generality, we assume that the normal field is directed inward $M_1$.

Let $\Psi$ be an eigenfunction corresponding to the first eigenvalue of the hypersurface $\Sigma$, normalized by the condition $\|\Psi\|_{L^{2}(\Sigma)}=1$. We define the harmonic extension $u$ of the function $\Psi$ to $M_1$, i.e. the unique solution of the Cauchy problem 
\begin{equation}\label{Diffur}
    \begin{cases}
    \Delta u = 0  |_{M_1}, \\
    u=\Psi  |_{\Sigma}.
    \end{cases}
\end{equation}

Substitute $u$ into the Rayleigh formula. We use the minimality of $\Sigma$ and the fact that we chose the orientation so that $$ - \int_{\Sigma} B(\grad^{\Sigma} u , \grad^{\Sigma} u)ds^g \geqslant 0,$$ and we get that

\begin{equation}
\begin{gathered}
\begin{gathered}
\label{Formule12}
\int_{M_1}((\Delta u)^2- \|\Hess u\|^2)dv^g \geqslant   cn\int_{M_1}|\grad u|^2dv^g  + \int_\Sigma(\Delta^{\Sigma} u)u_\nu ds^g + \\
+ \int_{\Sigma} \langle \grad^{\Sigma} u , \grad^{\Sigma} u_\nu \rangle ds^g  = cn\int_{M_1}|\grad u|^2dv^g + 2 \int_\Sigma(\Delta^{\Sigma}u )u_\nu ds^g =
\\ = nc\int_{M_1}|\grad u|^2dv^g + 2\lambda_1\int_\Sigma u_\nu u ds^g.
\end{gathered}
\end{gathered}
\end{equation}

Using Green's formula and the equation $\Delta u = 0 |_{M_1} $, we obtain that
\begin{equation}
\label{Formule13}
\int_{\Sigma}u_\nu u ds^g = \int_{M_1} -|\grad u|^2 + u\Delta u dv^g= -\int_{M_1} |\grad u|^2 dv^g.
\end{equation} 
Substituting this into the previous inequality, we obtain that
\begin{equation}
\label{Nain}
2\left(\lambda_1 - \frac{cn}{2}\right) \int_{M_1} |\grad u|^2 dv^g \geqslant \int_{M_1} \|\Hess u\|^2 dv^g \geqslant 0.
\end{equation}
Finally, we obtain the inequality first proved in \cite{ChoiWang},
$$\lambda_1(\Sigma) \geqslant \frac {cn}{2}.$$

\subsection{Proof of the first auxiliary proposition}
We need to introduce several constants. Unlike the work  \cite{SyreSpruckDuncan}, we do not try  to introduce them immediately, but only step by step.

Recall that the hypersurface $\Sigma_{t}$ is embedded for $|t| \in [0;\frac{1}{\sqrt{c}}\arctan(\sqrt{c}\,\Lambda^{-1}))$. Consider a parameter $\varepsilon \in (0;\frac{\Lambda}{2})$, we will fix its exact value later. Let $D_{\varepsilon}=\frac{1}{\sqrt{b}}\arctan(\sqrt{c}\,\varepsilon\,\Lambda^{-2})$. Then $D_{\varepsilon} < \frac{1}{\sqrt{c}}\arctan(\sqrt{c}\,\Lambda^{-1})$, and for any $ t \in [-D_{\varepsilon};D_{\varepsilon}]$ the surface
$\Sigma_{t}$ is an embedded submanifold.

\newtheorem*{lemma2}{Lemma 3.2.1}
\begin{lemma2}
[on the upper bound for the mean curvature]For any $ t \in [-D_{\varepsilon};D_{\varepsilon}]$ the mean curvature of the surface $\Sigma_{t}$ can be  estimated as follows,
$$H_{\Sigma_{t}} \leqslant \left(\frac{\varepsilon}{\sqrt{b} - \sqrt{c}\frac{\varepsilon}{\Lambda}}\right)\left(1 + \frac{bn}{\Lambda^2}\right)\sqrt{c}.$$
\end{lemma2}
Denote the RHS by $\widetilde{\varepsilon}$.
\begin{proof}
Taking a sum  in formula \eqref{Napp} over all $i$ we get that
\begin{equation}
\begin{gathered}
\begin{gathered}
\label{Formule3}
H_{\Sigma_{t}}  \leqslant  \sum_{i=1}^n \frac { \varkappa_{i} + \sqrt{b}\,\tan(\sqrt{b}\,t)}{1- \frac{1}{\sqrt{b}}\tan(\sqrt{b}\,t)\varkappa_{i}} \leqslant   \,H_{\Sigma} + \sum_{i=1}^n \frac{\frac{1}{\sqrt{b}}(b+\,\varkappa_{i}^2)\tan(\sqrt{b}\, t)}{1- \frac{1}{\sqrt{b}}\tan(\sqrt{b}\,t)\varkappa_{i}} \leqslant \\ \leqslant \frac{(bn + \Lambda^2)\tan(\sqrt{b} \, t)}{\sqrt{b} - \sqrt{c}\,\frac{\varepsilon}{\Lambda}} \leqslant 
\frac{1}{  \sqrt{b} - \sqrt{c}\,\frac{\varepsilon}{\Lambda}} \cdot \Lambda^2 \left(1+ \frac{bn}{\Lambda^2}\right)\cdot \tan(t\sqrt{b}) \leqslant \\ \leqslant \left(\frac{\varepsilon}{\sqrt{b} - \sqrt{c}\,\frac{\varepsilon}{\Lambda}}\right)\left(1 + \frac{bn}{\Lambda^2}\right)\sqrt{c}.
\end{gathered}
\end{gathered}
\end{equation}
\end{proof}

Denote $M_1^t$ = $\{p\in M_1 | d(p, \Sigma) \geqslant t\}$.
Note that for $t \in [0;D_\varepsilon]$ the boundary of the domain $M_1^t$ is $\Sigma_t$, i.e. a smooth embedded submanifold in $M$.

We present the following Lemma without proof, since it is similar to the proof of \cite[Lemma 3.6]{SyreSpruckDuncan}.
\newtheorem*{lemma3}{Lemma 3.2.2}
\begin{lemma3}
Let $0<\varepsilon<\frac{\Lambda}{2}$, let $v$ be a smooth function defined on $\overline{M_1}$, and let $\beta>0$ be a positive number and $t \in [0;D_\varepsilon]$. Then the following inequality holds,
 \begin{equation}
\begin{gathered}
\label{FormuleSire}
\int_{\Sigma}|\grad v|^2ds^g \leqslant \int_{\Sigma_{t}} |\grad v|^2ds^g  + \\ +(\widetilde{\varepsilon} + \beta) \int_{M_1 \setminus M_1^t} |\grad v|^2 dv^g + \beta^{-1}
\int_{M_1 \setminus M_1^t}|\Hess v|^{2} dv^g,
\end{gathered}
\end{equation}
\end{lemma3}

where  $\widetilde{\varepsilon}$ in \eqref{FormuleSire} is defined below Lemma 3.2.1.

\newtheorem*{lemma4}{Lemma 3.2.3}
\begin{lemma4}
[on the estimate of the integral over a surface by the integral over a domain]
The function $u$ defined by equation \eqref{Diffur} satisfies the following inequality,
\begin{equation}
\label{Formule12}
\sqrt{2nc}\int_{M_1}|\grad u|^2dv^g \leqslant \int_{\Sigma}|\grad u|^2ds^g.
\end{equation}
\begin{proof}
Recall that we consider the first non-constant eigenfunction $\Psi$ on the surface $\Sigma$, whose extension to the domain $M_1$ is $u$. Note that we consider $\Psi$  normalized as $\|\Psi\|_{L^2(\Sigma)} = 1$. We denote the unit normal field on $\Sigma$ directed inward  $M_1$ by $\nu$. Using  Green's formula and  the Cauchy-Schwarz inequality, we obtain
 \begin{equation}
 \begin{gathered}
\label{FormuleIneq1}
\left(\int_{M_1}|\grad u|^2dv^g\right)^2 = \left(\int_{\Sigma}u_\nu uds^g\right)^2  \leqslant \left(\int_{\Sigma}u_\nu^2 ds^g\right) \left(\int_{\Sigma} u^2 ds^g\right) = \\ = \int_{\Sigma}u_\nu^2 ds^g.
\end{gathered}
\end{equation}
On the other hand, we have the equality
 \begin{equation}
\label{FormuleEq1}
 \int_{\Sigma}u_\nu^2 ds^g =  \int_{\Sigma}|\grad u|^2 ds^g - \int_{\Sigma} |\grad^{\Sigma}u|^2ds^g  = \int_{\Sigma}|\grad u|^2 ds^g - \lambda_1,
\end{equation}
where the last equality follows from the variational description of the eigenvalues.

Substituting  equality \eqref{FormuleEq1} into  inequality \eqref{FormuleIneq1} and applying $a^2 + b^2 \geqslant 2ab,$ we obtain that
 \begin{equation}
 \begin{gathered}
\label{Formulip}
\int_{\Sigma}|\grad u|^2ds^g \geqslant \lambda_1+\left(\int_{M_1}|\grad u|^2 dv^g\right)^2 \geqslant \\ \geqslant  2\sqrt{\lambda_1}\int_{M_1}|\grad u|^2 dv^g \geqslant \sqrt{2nc}\int_{M_1}|\grad u|^2dv^g.
\end{gathered} 
\end{equation}
The last inequality in \eqref{Formulip} follows from the estimate by Choi and Wang. This finishes the proof.
\end{proof}
\end{lemma4}
Combining the results of the previous two Lemmas and using the idea that the integral of a non-negative function only increases as the domain increases, we obtain that
\begin{equation}
\label{AfterLemm}
(\sqrt{2nc} -\widetilde{\varepsilon}-\beta)\int_{M_1}|\grad u|^2dv^g \leqslant \int_{\Sigma_t}|\grad u|^2ds^g + \beta^{-1}\int_{M_1}|\Hess  u|^2dv^g.
\end{equation}

Let  $\gamma := \sqrt{2nc} -\widetilde{\varepsilon}-\beta$, $\delta := \frac{n\,c}{\sqrt{b}}\arctan(\frac{\varepsilon}{n\sqrt{c}})$ and $T:= \frac{\delta}{2\Lambda^2}$. Let us fix $\beta$ and $\varepsilon$ such that $\gamma > 0$. Then $\varepsilon, \widetilde{\varepsilon}, \beta, \gamma, \delta, T$ are now constants. Remark that at this moment the choice of $\beta$ and $\varepsilon$ is not unique. Latter in the proof of theorem 1.1.1 we choose some exact value of these constants. 

\newtheorem*{theorem2}{Proposition 3.2.4}
\begin{theorem2}
[first auxiliary proposition]
Let u be a function defined by  equation \eqref{Diffur}, $\beta, \gamma$ and $\delta$ be constants defined above. Then the following inequality holds,
\begin{equation}
\label{FormuleFirstPr}
\int_{M_1}|\grad u|^2dv^g \leqslant \frac{2\Lambda^2}{\delta \gamma} \int_{M_1^T\setminus M_1^{2T}}|\grad u|^2dv^g + \frac{1}{\gamma\cdot\beta}\int_{M_1}|\Hess  u|^2dv^g.
\end{equation}
\begin{proof}
First, we remark that the arctangent satisfies the following inequality, $$k\arctan(y) \leqslant \arctan(ky),$$ for  $y > 0$ and for $k\leqslant 1$. Thus, if $\Lambda \geqslant \sqrt{nc}$, then the estimate $\frac{\delta}{\Lambda^2} \leqslant \frac{1}{\sqrt{b}}\arctan(\frac{\varepsilon \sqrt{c}}{\Lambda^2}) = D_{\varepsilon}$ holds. Hence, for any $t \in [T;2T]$ formula \eqref{AfterLemm} also holds.  We integrate inequality \eqref{AfterLemm} over $[T;2T]$. Since $\Sigma_t$ is an embedded submanifold, Fubini's theorem implies that
\begin{equation}
\label{Formule12}
\frac{\delta \gamma}{2\Lambda^2}\int_{M_1}|\grad u|^2dv^g \leqslant \int_{M_1^T\setminus M_1^{2T}}|\grad u|^2dv^g + \frac{\delta}{2\Lambda^2\cdot \beta}\int_{M_1}|\Hess  u|^2dv^g.
\end{equation}

Now inequality  \eqref{FormuleFirstPr}  follows immediately .

\end{proof}
\end{theorem2}

\subsection{Proof of the second auxiliary proposition}

\newtheorem*{lemma5}{Lemma 3.3.1}
\begin{lemma5}
Let $M$ be a closed Riemannian manifold with Ricci curvature bounded below by 
\begin{equation}
\label{ricric}
\Ricc_M \geqslant cn. 
\end{equation}
For a harmonic function $w$ on $M$, we have the following inequality 
\begin{equation}
\label{Formula23}
-\Delta|\grad w|^2 \geqslant 2|\Hess w|^2 + 2cn|\grad w|^2.
\end{equation}
\begin{proof}
Let us write the Bochner's formula
\begin{equation}
\label{Formula22}
-\Delta|\grad w|^2 = -2\langle \grad \Delta w,\grad w\rangle + 2|\Hess w|^2 + 2\Ricc(\grad w,\grad w).
\end{equation}
Then  the bound \eqref{ricric} for the Ricci curvature imply the desired inequality.
\end{proof}
\end{lemma5}

\newtheorem*{lemma6}{Lemma 3.3.2}
\begin{lemma6}
Let $ \Omega \subset M$ where $M$ is a manifold with Ricci curvature bounded below, $\Ricc_M \geqslant cn > 0$, and $w$ is a smooth harmonic function in the domain $\Omega$. We denote by $\Omega^t$ the set $\{p\in \Omega | d(p,\partial \Omega) \geqslant t\}$. Note that for a sufficiently small $t$ the set $\Omega^t$ is domain with a smooth boundary. Further, we suppose that $t$ is small enough for $\Omega^t$ to be a domain with smooth boundary. 
Then the following inequality holds,
\begin{equation}
\label{Formule227}
\int_{\Omega^{2t}}|\grad w|^2dv^g   \leqslant \frac{1}{cn-1}\cdot t^{-2} \int_\Omega |\Hess w|^2dv^g.
\end{equation}
\begin{proof}
Let us take a smooth non-negative function $\rho^2$. We multiply the inequality from Lemma 3.3.1 by $\rho^2$. We integrate it and obtain
\begin{equation}
\label{Formula224}
\begin{gathered}
\begin{gathered}
\int_{\Omega}\rho^2(|\Hess w|^2 + cn|\grad w|^2)dv^g \leqslant -\frac{1}{2}\int_{\Omega}\rho^2 \Delta|\grad w|^2dv^g = \\ 
 = -\int_\Omega \rho \langle \grad \rho,\grad |\grad w|^2\rangle dv^g = -2\int_\Omega \rho \Hess w(\grad w,\grad \rho) dv^g \leqslant \\
\leqslant \int_\Omega\rho^2|\grad w|^2dv^g + \int_\Omega |\grad \rho|^2|\Hess w|^2dv^g.
\end{gathered}
\end{gathered}
\end{equation}
This inequality could be rewritten as 
\begin{equation}
(cn-1)\int_{\Omega}\rho^2|\grad w|^2dv^g + \int_\Omega(\rho^2 - |\grad \rho|^2)|\Hess w|^2dv^g  \leqslant 0.
\end{equation}
This implies
\begin{equation}
(cn-1)\int_{\Omega}\rho^2|\grad w|^2dv^g -  \int_\Omega |\grad \rho|^2|\Hess w|^2dv^g  \leqslant 0.
\end{equation}
It follows that
\begin{equation}
\label{Formula225}
\int_{\Omega}\rho^2|\grad w|^2dv^g   \leqslant \frac{1}{cn-1}\int_\Omega |\grad \rho|^2|\Hess w|^2dv^g.
\end{equation}
Finally, we choose $\rho$ such that $\rho \equiv 0|_{\Omega\setminus\Omega^t}$, $\rho \equiv 1|_{\Omega^{2t}}$ and \\
$|\grad(\rho)| \leqslant (1+\epsilon)t^{-1}$, then
\begin{equation}
\label{Formule226}
\int_{\Omega^{2t}}|\grad w|^2dv^g   \leqslant \frac{(1+\epsilon)^2}{cn-1}\cdot t^{-2} \int_\Omega |\Hess w|^2dv^g.
\end{equation}
Taking the limit as $\epsilon \to 0$, we get
\begin{equation}
\label{Formule227}
\int_{\Omega^{2t}}|\grad w|^2dv^g   \leqslant \frac{1}{cn-1}\cdot t^{-2} \int_\Omega |\Hess w|^2dv^g.
\end{equation}
\end{proof}
\end{lemma6}

\newtheorem*{theorem3}{Proposition 3.3.3}
\begin{theorem3}
[second auxiliary proposition]
Let u be a function defined by equation \eqref{Diffur}, and let $\delta,\gamma$ and $\Lambda$ be constants defined above. Then the following estimate of the integral holds,
\begin{equation}
\label{SecondBond}
\int_{M_1^{T}}|\grad w|^2dv^g \leqslant \frac{1}{cn-1}\cdot \frac{16\Lambda^2}{\delta^2} \int_{M_1} |\Hess w|^2dv^g.
\end{equation}
\begin{proof}
We apply Lemma 3.3.2. Take $\Omega=M_1, w=u$ and
$t=\frac{T}{2}(T=\frac{\delta}{2\Lambda^2})$, then by Lemma 3.3.2 we obtain the inequality \eqref{SecondBond}.
\end{proof}
\end{theorem3}

\subsection{Proof of Theorem 1.1.1}
Recall that we consider  a compact orientable Riemannian manifold $M$ of dimension  $n+1$  without boundary, $M$  satisfies condition \eqref{Condition}, and $\Sigma$ is a closed orientable hypersurface minimally embedded in $M$.

We defined the harmonic extension $u$ of the first eigenfunction $\Psi$ of the hypersurface $\Sigma$ to $M_1$ as the unique solution of the problem \eqref{Diffur}, where $M_1$ is one of the connected components of $M\setminus \Sigma$.

For the function $u$, we obtained two estimates \eqref{FormuleFirstPr} and \eqref{SecondBond}. Substituting \eqref{SecondBond} into \eqref{FormuleFirstPr}, we have
\begin{equation}
\label{Formule231}
\int_{M_1}|\grad u|^2dv^g \leqslant \left(\frac{32\Lambda^6}{(cn-1)\delta^3 \gamma}+ \frac{1}{\beta\gamma}\right) \int_{M_1}|\Hess u|^2dv^g.
\end{equation}
This estimate of the Hessian via the gradient permits us to improve Choi-Wang result as we promised in Subsection 3.1. Substitution of \eqref{Formule231} into $\eqref{Nain}$ yields that
\begin{equation}\label{Formule242}
2\left(\lambda_1 - \frac{cn}{2}\right) \int_{M_1} |\grad u|^2 dv^g \geqslant \frac{1}{\frac{32\Lambda^6}{(cn-1)\delta^3 \gamma}+ \frac{1}{\beta\gamma}} \int_{M_1}|\grad u|^2dv^g.
\end{equation}
Dividing both sides by $\int_{M_1}|\grad u|^2dv^g$, we get
\begin{equation}\label{Formule243}
\lambda_{1}(\Sigma) \geqslant \frac{cn}{2} + \frac{1}{\frac{64\Lambda^6}{(cn-1)\delta^3 \gamma}+ \frac{2}{\beta\gamma}}.
\end{equation}
To obtain the estimate \eqref{Main}, we introduce constants
\begin{align}
  \begin{aligned}
  a_n=\frac{(cn-1)\delta^3\gamma}{64},
  \end{aligned}
  &&
  \begin{aligned}
  b_n=\frac{(cn-1)\delta^3}{32\beta},
  \end{aligned}
 \end{align}
and fix
\begin{align}
  \begin{aligned}
    \varepsilon= \frac{2(2\sqrt{b}-\sqrt{c})\sqrt{n}}{\sqrt{b}\frac{bc+1}{b}3},
  \end{aligned}
  &&
  \begin{aligned}
  \beta=\frac{\sqrt{nc}}{20}.
  \end{aligned}
 \end{align}

Then $\gamma\geqslant \frac{3\sqrt{nc}}{100}$, and our main constants are estimated as follows,
\begin{equation}\label{Formule233}
a_n \geqslant \frac{3(cn-1)(c \cdot n)^{7/2}}{b^{3/2}\,6400}\arctan^3\left(\frac{\varepsilon}{\sqrt{c}\, n}\right),
\end{equation}
\begin{equation}\label{Formule234}
b_n \leqslant \frac{5(cn-1)(c\cdot n)^{5/2}}{b^{3/2}\,8}\arctan^3\left(\frac{\varepsilon}{\sqrt{c}\, n}\right).
\end{equation}
As a result, we can write the final estimate
\begin{equation}\label{Main2}
\lambda_{1}(\Sigma) \geqslant \frac{cn}{2} + \frac{a_n}{\Lambda^6 + b_n}.
\end{equation}
\begin{center}
\textbf{Acknowledgment}
\end{center}
The author is grateful to Alexei Penskoi for attaching his attention to this problem, fruitful discussions and help in preparation of this paper.

\bibliographystyle{alpha}
\bibliography{LowerBound.org.tug.org.tug}
\end{document}